\documentclass[11pt]{amsart}
\usepackage{amscd,amsmath,amssymb,enumerate,amssymb,amsfonts}
\usepackage{comment}
\usepackage{hyperref}
 \usepackage{xcolor}
\usepackage{tikz,pgfplots,pgf}
\tikzset{node distance=2cm, auto}
\usetikzlibrary{calc}
\usetikzlibrary{trees}

\newtheorem{theorem}{Theorem}[section]
\newtheorem{proposition}[theorem]{Proposition}
\newtheorem{definition}[theorem]{Definition}
\newtheorem{corollary}[theorem]{Corollary}

\def\F{\mathcal{F}}
\def\G{\mathcal{G}}
\def\H{\mathcal{H}}
\def\I{\mathcal{I}}
\def\K{\mathcal{K}}
\def\W{\mathcal{W}}
\def\L{\mathcal{L}}

\def\R{\mathcal{R}}
\def\S{\mathcal{S}}
\def\RN{\mathcal{RN}}
\def\AS{\mathcal{A}}

\def\C{\mathbb{C}}

\def\N{\mathbb{N}}

\def\co{\mathrm{co}}
\def\aco{\mathrm{aco}}

\def\lin{\mathrm{lin}}
\def\rank{\mathrm{rank}}

\usepackage{vmargin}
\setpapersize{A4}
\setmargins{2.5cm}
{1.5cm}
{16.5cm}
{23.42cm}
{10pt}
{1cm}
{0pt}
{2cm}

\makeatletter
\@namedef{subjclassname@2020}{\textup{2020} Mathematics Subject Classification}
\makeatother

\begin{document}

\title[On holomorphic mappings with compact type range]{On holomorphic mappings with compact type range}

\author[A. Jim{\'e}nez-Vargas]{A. Jim{\'e}nez-Vargas}
\address[A. Jim{\'e}nez-Vargas]{Departamento de Matem{\'a}ticas, Universidad de Almer{\'i}a, 04120, Almer{\'i}a, Spain}
\email{ajimenez@ual.es}

\author[D. Ruiz-Casternado]{D. Ruiz-Casternado}
\address[D. Ruiz-Casternado]{Departamento de Matem{\'a}ticas, Universidad de Almer{\'i}a, 04120, Almer{\'i}a, Spain}
\email{davidrc3005@gmail.com}

\author[J. M. Sepulcre]{J. M. Sepulcre}
\address[J. M. Sepulcre]{Departamento de Matem\'aticas, Universidad de Alicante, 03080 Alicante, Spain}
\email{JM.Sepulcre@ua.es}

\date{\today}


\thanks{This paper was finished while the first author was visiting the Department of Mathematics at Alicante University (Spain). He thanks this Department for its kind hospitality. The first author was partially supported by project UAL-FEDER grant UAL2020-FQM-B1858, by Junta de Andaluc\'{\i}a grants P20$\_$00255 and FQM194, and by Ministerio de Ciencia e Innovaci\'on grant PID2021-122126NB-C31. The third author was also supported by PGC2018-097960-B-C22 (MCIU/AEI/ERDF, UE)}

\begin{abstract}
Using Mujica's linearization theorem, we extend to the holomorphic setting some classical characterizations of compact (weakly compact, Rosenthal, Asplund) linear operators between Banach spaces such as the Schauder, Gantmacher and Gantmacher--Nakamura theorems and the Davis--Figiel--Johnson--Pe\l czynski, Rosenthal and Asplund factorization theorems.
\end{abstract}
\maketitle

\section*{Introduction}\label{section 0}

S. Dineen \cite[p. 417]{Din-81} pointed out that a remarkable result due to K. F. Ng \cite{Ng-71} provides a Banach space $\G^\infty(U)$ whose dual is isometrically isomorphic to the Banach space $\H^\infty(U)$ of all bounded holomorphic complex-valued functions on an open subset $U$ of an arbitrary complex Banach space $E$, endowed with the supremum norm.

J. Mujica presented in \cite[Theorem 2.1]{Muj-91} a linearization theorem for bounded holomorphic mappings which is a refinement of Dineen's idea. Specifically, he proved that there exist a Banach space $\G^\infty(U)$ and a bounded holomorphic mapping $\delta_U\colon U\to\G^\infty(U)$ satisfying the following universal property: for each complex Banach space $F$ and each bounded holomorphic mapping $f\colon U\to F$, there exists a unique bounded linear operator $T_f\colon\G^\infty(U)\to F$ such that $T_f\circ\delta_U=f$. R. Ryan \cite{Rya-80} had previously obtained a polynomial version of Mujica's theorem by using a tensor product approach. In \cite{GalGarMae-92}, P. Galindo, D. Garc\'ia and M. Maestre established a linearization theorem for holomorphic mappings of bounded type. Moreover, J. Mujica and L. Nachbin also obtained in \cite{Muj-92} a linearization theorem for holomorphic functions between locally convex spaces.

Applying Ng's theorem, Mujica defined the space $\G^\infty(U)$ as the norm-closed linear subspace of $\H^\infty(U)^*$ formed by those functionals which are $\tau_c$-continuous when restricted to the closed unit ball of $\H^\infty(U)$, where $\tau_c$ denotes the compact-open topology on $\H^\infty(U)$. The correspondence $f\mapsto T_f$ is an isometric isomorphism between the space $\H^\infty(U,F)$ of all bounded holomorphic mappings from $U$ into $F$ with the supremum norm, and the space $\L(G^\infty(U);F)$ of all continuous linear operators from $\G^\infty(U)$ into $F$ with the operator norm. The mapping $\delta_U\colon U\to\G^\infty(U)$ is defined by $\delta_U(x)=\delta(x)$, where $\delta(x)$ is the evaluation functional at $x$ defined on $\H^\infty(U)$. The proof of Theorem 2.1 in \cite{Muj-91} shows that the closed unit ball of $\G^\infty(U)$ coincides with the norm-closed absolutely convex hull of $\delta_U(U)$. In particular, $\G^\infty(U)$ is the norm-closed linear hull of $\delta_U(U)$ in $\H^\infty(U)^*$.

In Section 3 of his paper \cite{Muj-91}, Mujica briefly dealt with holomorphic mappings that have compact type range. To be more precise, he showed that a mapping $f\in\H^\infty(U,F)$ has relatively compact range (respectively, relatively weakly compact range, finite rank) if and only if its linearization $T_f\in\L(G^\infty(U);F)$ is a compact (respectively, weakly compact, bounded finite-rank) operator. Our purpose in this paper is to complete the study initiated by Mujica.

From a local point of view, the properties of compactness, weak compactness, Rosenthal and Asplund for holomorphic mappings from $U$ into $F$ were addressed by R. M. Aron and M. Schottenloher \cite{AroSch-76}, R. Ryan \cite{Rya-88}, M. Lindstr\"om \cite{Lin-89} and N. Robertson \cite{Rob-92}, respectively. Let us recall that a mapping $f\colon U\to F$ is said to be \textit{locally compact (respectively, locally weakly compact, locally Rosenthal, locally Asplund)} if every point $x\in U$ has a neighborhood $V_x\subseteq U$ such that $f(V_x)$ is relatively compact (respectively, relatively weakly compact, Rosenthal, Asplund) in $F$. Clearly, every mapping $f\colon U\to F$ having relatively compact range (respectively, relatively weakly compact range, Rosenthal range, Asplund range) is locally compact (respectively, locally weakly compact, locally Rosenthal, locally Asplund), however, the converse is not true in general for mappings $f\in\H^\infty(U,F)$ (see Example 3.2 in \cite{Muj-91} for the first two types of mappings).


We have organized this note as follows. In Section \ref{section 1}, we recall Mujica's linearization theorem and some of its consequences that will be needed to establish our results. With the aid of the notion of transpose mapping of a bounded holomorphic mapping, Section \ref{section 2} is devoted to the analogues for bounded holomorphic mappings of the results due to Schauder, Gantmacher and Nakamura on the compactness and weak compactness of the adjoint of a bounded linear operator between Banach spaces. We suggest the reader to compare our results with Schauder and Gantmacher type theorems for holomorphic mappings of bounded type established by M. Gonz\'alez and J. M. Guti\'errez \cite{GonGut-93,GonGut-96,GonGut-97} and R. Ryan \cite{Rya-80,Rya-88}.

As a main result, we show that some factorization theorems for bounded linear operators between Banach spaces can be extended to the holomorphic setting. This is the case of the Davis--Figiel--Johnson--Pe\l czynski factorization theorem \cite{dfjp} which states that every weakly compact operator factors through a reflexive Banach space, the Rosenthal factorization theorem (see \cite{Alv-88}) which asserts that every Rosenthal operator factors through a Banach space not containing $\ell_1$, and the Asplund factorization theorem which assures that every Asplund operator factors through an Asplund space (see \cite[Theorem 5.3.5]{Bou-83}).

We refer to the book of R. E. Megginson \cite{Meg-98} for a complete study on weak topologies and linear operators on Banach spaces; to the monograph of J. Mujica \cite{Muj-86} for the theory of holomorphic mappings on Banach spaces; and to the book of A. Pietsch \cite{Pie-80} for the theory of operator ideals.

\bigskip

\textbf{Notation.} Through the paper, given a complex Banach space $E$, we denote by $B_E$, $\stackrel{\circ}{B}_E$, $S_E$ and $E^*$ the closed unit ball, the open unit ball, the unit sphere and the dual space of $E$, respectively. For a set $A\subseteq E$, $\lin(A)$, $\overline{\lin}(A)$, $\co(A)$, $\overline{\co}(A)$, $\aco(A)$ and $\overline{\aco}(A)$ stand for the linear hull, the norm-closed linear hull, the convex hull, the norm-closed convex hull, the absolutely convex hull, the norm-closed absolutely convex hull of $A$ in $E$, respectively. If $E$ and $F$ are locally convex Hausdorff spaces, $\L(E;F)$ denotes the vector space of all continuous linear operators from $E$ into $F$. Unless stated otherwise, if $E$ and $F$ are Banach spaces, we will understand that they are endowed with the norm topology. Given $T\in\L(E;F)$, $T^*\colon F^*\to E^*$ denotes the adjoint operator of $T$. 

We will sometimes use the following notation: for each $x\in E$ and $x^*\in E^*$, $\left\langle x^*,x\right\rangle$ is defined to be $x^*(x)$.

\section{Preliminaries}\label{section 1}

Let $U$ be an open subset of a complex Banach space $E$. We know that $\H^\infty(U)$ is a Banach space under the supremum norm and it is actually a dual Banach space. In fact, there are different ways to construct a predual of $\H^\infty(U)$. The most straightforward one is as the norm-closed linear subspace of $\H^\infty(U)^*$ generated by the functionals $\delta(x)\in\H^\infty(U)^*$ with $x\in U$, defined by
$$
\left\langle\delta(x),f\right\rangle=f(x)\qquad \left(f\in\H^\infty(U)\right).
$$
Theorem 2.1 in \cite{Muj-91} justifies the following notation.

\begin{definition}
Let $U$ be an open subset of a complex Banach space $E$. The space $\G^\infty(U)$ is the norm-closed linear subspace of $\H^\infty(U)^*$ given by $\overline{\lin}\left\{\delta(x)\colon x\in U\right\}$.
\end{definition}

For consistency with the preceding notation, we may consider the mapping $\delta_U\colon U\to\G^\infty(U)$ defined by $\delta_U(x)=\delta(x)$ for $x\in U$.

In \cite{Muj-91}, Mujica established the following properties of $\mathcal{G}^\infty(U)$ and $\delta_U$.

\begin{theorem}\label{teo1}\cite[Theorem 2.1]{Muj-91}
Let $E$ be a complex Banach space and let $U$ be an open set in $E$.
\begin{enumerate}
\item $\delta_U\colon U\to\G^\infty(U)$ is a holomorphic mapping with $\left\|\delta(x)\right\|=1$ for all $x\in U$.
\item For every complex Banach space $F$ and every mapping $f\in\H^\infty(U,F)$, there exists a unique operator $T_f\in\L(\G^\infty(U);F)$ such that $T_f\circ\delta_U=f$, that is, the diagram
$$
\begin{tikzpicture}
\node (U) {$U$};
\node (GU) [below of=U] {$\G^\infty(U)$};
\node (F) [right of=GU] {$F$};
\draw[->] (U) to node {$f$} (F);
\draw[->] (U) to node [swap] {$\delta_U$} (GU);
\draw[->, dashed] (GU) to node [swap] {$T_f$} (F);
\end{tikzpicture}
$$
commutes. Furthermore, $\left\|T_f\right\|=\left\|f\right\|_\infty$.
\item For every complex Banach space $F$, the mapping $J_{U,F}\colon f\mapsto T_f$ is an isometric isomorphism from $\H^\infty(U,F)$ onto $\L(\G^\infty(U);F)$. Its inverse $J^{-1}_{U,F}$ is the mapping $T\mapsto T\circ\delta_U$ from $\L(\G^\infty(U);F)$ onto $\H^\infty(U,F)$.
\item $\H^\infty(U)$ is isometrically isomorphic to $\G^\infty(U)^*$, via the mapping $J_U:=J_{U,\mathbb{C}}\colon\H^\infty(U)\to\G^\infty(U)^*$ given by
$$
J_U(f)=T_f\qquad (f\in\H^\infty(U)).
$$
As a consequence, we have
$$
\left\langle J_U(f),\delta_U(x)\right\rangle=f(x)\qquad (f\in\H^\infty(U),\; x\in U).
$$
\item The closed unit ball of $\G^\infty(U)$ coincides with the norm-closed absolutely convex hull of $\delta_U(U)$.$\hfill\Box$
\end{enumerate}
\end{theorem}

From Theorem \ref{teo1} (4), we immediately deduce that, on bounded subsets of $\H^\infty(U)$, the weak* topology agrees with the topology of pointwise convergence.

\begin{corollary}\label{cor-1}
Let $U$ be an open subset of a complex Banach space $E$. Let $(f_i)$ be a net in $\H^\infty(U)$ and $f\in\H^\infty(U)$.
\begin{enumerate}
	\item If $(f_i)\to f$ weak* in $\H^\infty(U)$, then $(f_i)\to f$ pointwise on $U$.
	\item It $(f_i)$ is bounded in $\H^\infty(U)$ and $(f_i)\to f$ pointwise on $U$, then $(f_i)\to f$ weak* in $\H^\infty(U)$.$\hfill\Box$
\end{enumerate}
\end{corollary}

Using Corollary \ref{cor-1}, we obtain the following result.

\begin{corollary}\label{cor-3}
Let $E$ and $F$ be complex Banach spaces, let $U$ and $V$ be open subsets of $E$ and $F$, respectively, and let $f\colon U\to V$ be a holomorphic mapping. Then there exists a unique operator $\widehat{f}\in\L(\G^\infty(U);\G^\infty(V))$ such that $\widehat{f}\circ\delta_U=\delta_V\circ f$. Furthermore, $||\widehat{f}||=1$.
\end{corollary}

\begin{proof}
Let $C_f\colon\H^\infty(V)\to\H^\infty(U)$ be the composition operator defined by
$$
C_f(g)=g\circ f\qquad (g\in\H^\infty(V)).
$$
Clearly, $C_f\in\L(\H^\infty(V);\H^\infty(U))$ with $||C_f||=1$. We claim that $J_U\circ C_f\circ J^{-1}_V$ is weak*-to-weak* continuous from $\G^\infty(V)^*$ into $\G^\infty(U)^*$. Let $(\phi_i)$ be a net in $\G^\infty(V)^*$ and $\phi\in\G^\infty(V)^*$. Assume that $(\phi_i)\to\phi$ weak* in $\G^\infty(V)^*$. By \cite[Corollary 2.6.10]{Meg-98}, $(\phi_i)$ is norm-bounded in $\G^\infty(V)^*$. Clearly, $(J^{-1}_V(\phi_i))\to J^{-1}_V(\phi)$ weak* in $\H^\infty(V)$. Hence $(J^{-1}_V(\phi_i))\to J^{-1}_V(\phi)$ pointwise on $V$ by Corollary \ref{cor-1} (1). In particular, $(J^{-1}_V(\phi_i)\circ f)\to J^{-1}_V(\phi)\circ f$ pointwise on $U$. Since the net $(J^{-1}_V(\phi_i)\circ f)$ is norm-bounded since $\left\|J^{-1}_V(\phi_i)\circ f\right\|_\infty\leq \left\|J^{-1}_V(\phi_i)\right\|_\infty=\left\|\phi_i\right\|$ for all $i$, Corollary \ref{cor-1} (2) guarantees that $(J^{-1}_V(\phi_i)\circ f)\to J^{-1}_V(\phi)\circ f$ weak* in $\H^\infty(U)$, that is, $(C_f(J^{-1}_V(\phi_i)))\to C_f(J^{-1}_V(\phi))$ weak* in $\H^\infty(U)$. Finally, $(J_U(C_f(J^{-1}_V(\phi_i))))\to J_U(C_f(J^{-1}_V(\phi)))$ weak* in $\G^\infty(U)^*$, and this proves our claim.

Now, by \cite[Corollaries 3.1.11 and 3.1.5]{Meg-98}, there is a unique operator $\widehat{f}\in\L(\G^\infty(U);\G^\infty(V))$ such that $(\widehat{f})^*=J_U\circ C_f\circ J^{-1}_V$. Clearly, $||\widehat{f}||=||(\widehat{f})^*||=||C_f||=1$. Given $g\in\H^\infty(V)$ and $x\in U$, we have
$$
\left\langle(\widehat{f})^*(J_V(g)),\delta_U(x)\right\rangle=\left\langle J_V(g)\circ\widehat{f},\delta_U(x)\right\rangle=\left\langle J_V(g),\widehat{f}(\delta_U(x))\right\rangle
$$
and
\begin{align*}
\left\langle (J_U\circ C_f\circ J^{-1}_V)(J_V(g)),\delta_U(x)\right\rangle&=\left\langle J_U(C_f(g)),\delta_U(x)\right\rangle
=\left\langle J_U(g\circ f),\delta_U(x)\right\rangle\\
&=g(f(x))=\left\langle J_V(g),\delta_V(f(x))\right\rangle .
\end{align*}
From above we infer that
$$
\left\langle J_V(g),\widehat{f}(\delta_U(x))\right\rangle=\left\langle J_V(g),\delta_V(f(x))\right\rangle \qquad (g\in\H^\infty(V),\; x\in U),
$$
the surjectivity of $J_V\colon\H^\infty(V)\to\G^\infty(V)^*$ yields
$$
\left\langle\phi,\widehat{f}(\delta_U(x))\right\rangle=\left\langle\phi,\delta_V(f(x))\right\rangle \qquad (\phi\in\G^\infty(V)^*,\; x\in U),
$$
and this implies that $\widehat{f}\circ\delta_U=\delta_V\circ f$.
\end{proof}

We finish this section with some results related to the transpose of a bounded holomorphic mapping.

Let $U$ be an open subset of a complex Banach space $E$, let $F$ be a complex Banach space and let $f\in\H^\infty(U,F)$. Given $\varphi\in F^*$, it is clear that $\varphi\circ f\colon U\to\C$ is holomorphic with
$$
\left|(\varphi\circ f)(x)\right|
\leq\left\|\varphi\right\|\left\|f(x)\right\|\leq\left\|\varphi\right\|\left\|f\right\|_\infty
$$
for all $x\in U$. Hence $\varphi\circ f\in\H^\infty(U)$ with $\left\|\varphi\circ f\right\|_\infty\leq\left\|\varphi\right\|\left\|f\right\|_\infty$. This justifies the following.

\begin{definition}
Let $U$ be an open subset of a complex Banach space $E$, let $F$ be a complex Banach space and $f\in\H^\infty(U,F)$. We will call the transpose mapping $f^t\colon F^*\to\H^\infty(U)$ defined by
$f^t(\varphi)=\varphi\circ f$ the transpose of f.
\end{definition}

Clearly, $f^t$ is linear and continuous with $||f^t||\leq\left\|f\right\|_\infty$. In fact, $||f^t||=\left\|f\right\|_\infty$. Indeed, for $0<\varepsilon<\left\|f\right\|_\infty$, take $x\in U$ such that $\left\|f(x)\right\|>\left\|f\right\|_\infty-\varepsilon$. By Hahn--Banach theorem, there exists $\phi\in F^*$ with $\left\|\phi\right\|=1$ such that $\phi(f(x))=\left\|f(x)\right\|$. We have
\begin{align*}
\left\|f^t\right\|&\geq\sup_{0\neq\varphi\in F^*}\frac{\left\|f^t(\varphi)\right\|_\infty}{\left\|\varphi\right\|}
\geq\frac{\left\|\phi\circ f\right\|_\infty}{\left\|\phi\right\|}=\left\|\phi\circ f\right\|_\infty\\
&\geq \left|\phi(f(x))\right|=\left\|f(x)\right\|>\left\|f\right\|_\infty-\varepsilon .
\end{align*}
Letting $\varepsilon\to 0$, one obtains $||f^t||\geq\left\|f\right\|_\infty$, as desired. Moreover, note that
\begin{align*}
\left\langle(J_U\circ f^t)(\varphi),\delta(x)\right\rangle&=\left\langle J_U(\varphi\circ f),\delta(x)\right\rangle\\
                                                          &=\varphi(f(x))=\left\langle\varphi,T_f(\delta(x))\right\rangle\\
																													&=\left\langle (T_f)^*(\varphi),\delta(x)\right\rangle
\end{align*}
for all $\varphi\in F^*$ and $x\in U$, and since $\G^\infty(U)=\overline{\lin}(\delta_U(U))$, we deduce that $J_U\circ f^t=(T_f)^*$, that is, $f^t=J^{-1}_U\circ (T_f)^*$. We have proved the following.

\begin{proposition}\label{prop-A}
Let $U$ be an open subset of a complex Banach space $E$, let $F$ be a complex Banach space and let $f\in\H^\infty(U,F)$. Then $f^t\in\L(F^*;\H^\infty(U))$ with $||f^t||=\left\|f\right\|_\infty$. Furthermore, $f^t=J^{-1}_U\circ (T_f)^*$. $\hfill\Box$
\end{proposition}

We next see that the mapping $f\mapsto f^t$ identifies $\H^\infty(U,F)$ with the subspace of $\L(F^*;\H^\infty(U))$ formed by all weak*-to-weak* continuous linear operators from $F^*$ into $\H^\infty(U)$ (see \cite[Corollary 3.1.12]{Meg-98}).

\begin{proposition}\label{teo-4-1}
Let $U$ be an open subset of a complex Banach space $E$ and let $F$ be a complex Banach space. The mapping $f\mapsto f^t$ is an isometric isomorphism from $\H^\infty(U,F)$ onto $\L((F^*,w^*);(\H^\infty(U),w^*))$.
\end{proposition}

\begin{proof}
Let $f\in\H^\infty(U,F)$. Clearly, $f^t=J^{-1}_U\circ (T_f)^*\in\L((F^*,w^*);(\H^\infty(U),w^*))$ by Theorem \ref{teo1} (2) and \cite[Theorem 3.1.11]{Meg-98}. 
Moreover, we have $||f^t||=\left\|f\right\|_\infty$ by Proposition \ref{prop-A}. It remains to show the surjectivity of the mapping in the statement. Take $T\in\L((F^*,w^*);(\H^\infty(U),w^*))$. Then $J_U\circ T\in\L((F^*,w^*);(\G^\infty(U)^*,w^*))$ 
and, by \cite[Theorem 3.1.11]{Meg-98}, there is a $S\in\L(\G^\infty(U);F)$ such that $S^*=J_U\circ T$. By Theorem \ref{teo1} (3), there exists $f\in\H^\infty(U,F)$ such that $T_f=S$. Hence $T=J^{-1}_U\circ(T_f)^*=f^t$, as desired.
\end{proof}

\section{Linearization of holomorphic mappings with compact type range}\label{section 2}

Let us recall that a bounded linear operator between Banach spaces $T\colon E\to F$ is said to be \textit{compact (separable, weakly compact, Rosenthal, Asplund)} if $T(B_E)$ is relatively compact (respectively, separable, relatively weakly compact, Rosenthal, Asplund) in $F$.

We denote by $\F(E,F)$, $\overline{\F}(E,F)$, $\K(E,F)$, $\S(E,F)$, $\W(E,F)$, $\R(E,F)$ and $\AS(E,F)$ the linear spaces of bounded finite-rank linear operators, approximable linear operators (i.e., operators which are the norm limits of bounded finite-rank operators), compact linear operators, bounded separable linear operators, weakly compact linear operators, Rosenthal linear operators and  Asplund linear operators
from $E$ into $F$, respectively. The following inclusions are known:
\begin{align*}
\F(E,F)&\subseteq\overline{\F}(E,F)\subseteq\K(E,F)\subseteq\W(E,F)\subseteq\R(E,F)\cap\AS(E,F),\\
\K(E,F)&\subseteq\S(E,F).
\end{align*}

Our aim is to study the following holomorphic variants of these concepts. If $U$ is an open subset of a complex Banach space $E$ and $F$ is a complex Banach space, we will consider bounded holomorphic mappings $f\colon U\to F$ that have a range $f(U)\subseteq F$ satisfying an algebraic or topological property as, for instance, finite-dimensional range, relatively compact range, separable range, relatively weakly compact range, Rosenthal range or Asplund range.


Note that if $T\in\L(E,F)$, then $T$ is a compact (respectively, separable, weakly compact, Rosenthal, Asplund) linear operator if and only if the holomorphic mapping $T|_{\stackrel{\circ}{B}_E}$ has relatively compact (respectively, separable, relatively weakly compact, Rosenthal, Asplund) range.

The study of the connections between these compactness properties of a mapping $f\in\H^\infty(U,F)$ and its corresponding associated operator $T_f\in\L(\G^\infty(U);F)$ was initiated by Mujica in Propositions 3.1 and 3.4 of \cite{Muj-91}. Apparently, these results are the only known on this question and they have been included here with their proofs for the convenience of the reader. We have divided our study for the different types of holomorphic mappings considered.

\subsection{Bounded finite-rank holomorphic mappings}

Let us recall (see \cite[p. 72]{Muj-91}) that a mapping $f\colon U\to F$ has \textit{finite rank} if $\lin(f(U))$ is a finite dimensional subspace of $F$ in which case this dimension is called the \textit{rank} of $f$ and denoted by $\rank(f)$. Let $\H^\infty_{\F}(U,F)$ denote the linear space of all bounded finite-rank holomorphic mappings from $U$ to $F$.

\begin{theorem}\label{teo-3-0}
Let $U$ be an open subset of a complex Banach space $E$, let $F$ be a complex Banach space and let $f\in\H^\infty(U,F)$. The following are equivalent:
\begin{enumerate}
	\item $f\colon U\to F$ has finite rank.
	\item $T_f\colon\G^\infty(U)\to F$ has finite rank.
	\item $f^t\colon F^*\to\H^\infty(U)$ has finite rank.
\end{enumerate}
In that case, $\rank(f)=\rank(T_f)=\rank((T_f)^*)=\rank(f^t)$.
\end{theorem}

\begin{proof}
Recall that $f(U)=T_f(\delta_U(U))$ by Theorem \ref{teo1} (2).

$(1)\Leftrightarrow(2)$ (\cite[Proposition 3.1 (a)]{Muj-91}): If $f$ has finite rank, then $\lin(f(U))$ is finite dimensional and therefore closed in $F$. We have
\begin{align*}
T_f(\G^\infty(U))&=T_f(\overline{\lin}(\delta_U(U)))\subseteq\overline{T_f(\lin(\delta_U(U)))}\\
                 &=\overline{\lin}(T_f(\delta_U(U)))=\overline{\lin}(f(U))=\lin(f(U))
\end{align*}
and hence $T_f$ has finite rank. Conversely, if $T_f$ has finite rank, then $f$ has finite rank since
\begin{align*}
\lin(f(U))&=\lin(T_f(\delta_U(U)))=T_f(\lin(\delta_U(U)))\\
          &\subseteq T_f(\overline{\lin}(\delta_U(U)))=T_f(\G^\infty(U)).
\end{align*}
Furthermore, in this case we have $\rank(f)=\rank(T_f)$.

$(2)\Leftrightarrow(3)$: Since the space $\F(E,F)$ of bounded finite-rank linear operators is a completely symmetric operator ideal (see \cite[4.4.7]{Pie-80}), we have
\begin{align*}
T_f\in\F(\G^\infty(U),F)&\Leftrightarrow (T_f)^*\in\F(F^*,\G^\infty(U)^*)\\
&\Leftrightarrow f^t=J^{-1}_U\circ(T_f)^*\in\F(F^*,\H^\infty(U)).
\end{align*}
In this case we now have $\rank(T_f)=\rank((T_f)^*)=\rank(f^t)$.
\end{proof}

\subsection{Holomorphic mappings with relatively compact range}

We denote by $\H^\infty_{\K}(U,F)$ the linear space of all holomorphic mappings from $U$ to $F$ that have relatively compact range. The equivalence $(1)\Leftrightarrow (3)$ of the next result is a version for holomorphic mappings with relatively compact range of the classical Schauder's theorem on the relationship of the compactness of a bounded linear operator between Banach spaces and its adjoint.

\begin{theorem}\label{teo-4-2}
Let $U$ be an open subset of a complex Banach space $E$, let $F$ be a complex Banach space and let $f\in\H^\infty(U,F)$. The following are equivalent:
\begin{enumerate}
  \item $f\colon U\to F$ has relatively compact range.
	\item $T_f\colon\G^\infty(U)\to F$ is compact.
	\item $f^t\colon F^*\to\H^\infty(U)$ is compact.
	\item $f^t\colon F^*\to\H^\infty(U)$ is bounded-weak*-to-norm continuous.
	\item $f^t\colon F^*\to\H^\infty(U)$ is compact and bounded-weak*-to-weak continuous.
	\item $f^t\colon F^*\to\H^\infty(U)$ is compact and weak*-to-weak continuous.
\end{enumerate}
\end{theorem}

\begin{proof}
$(1)\Leftrightarrow(2)$ (\cite[Proposition 3.1 (b)]{Muj-91}): 
Since $\delta_U(U)\subseteq B_{\G^\infty(U)}$, if $T_f\in\K(\G^\infty(U),F)$, then $T_f(\delta_U(U))$ must be relatively compact in $F$, and therefore $f\in\H^\infty_{\K}(U,F)$ since $f(U)=T_f(\delta_U(U))$ by Theorem \ref{teo1} (2). Conversely, since $B_{\G^\infty(U)}=\overline{\aco}(\delta_U(U))$ by Theorem \ref{teo1} (5), one has
$$
T_f(B_{\G^\infty(U)})=T_f(\overline{\aco}(\delta_U(U)))\subseteq\overline{\aco}(T_f(\delta_U(U)))\subseteq\overline{\aco}(\overline{T_f(\delta_U(U))}).
$$
So, if $f\in\H^\infty_{\K}(U,F)$, then $T_f(\delta_U(U))$ is relatively compact in $F$, hence $\overline{\aco}(\overline{T_f(\delta_U(U))})$ is compact in $F$ by Mazur's compactness theorem (\cite[Theorem 2.8.15]{Meg-98}) and the fact that $\aco(A)=\co(\mathbb{D}A)$ for any subset $A$ of a normed space $E$, where $\mathbb{D}$ denotes the closed unit disc of $\mathbb{C}$. Therefore $T_f(B_{\G^\infty(U)})$ is relatively compact in $F$, which means that $T_f\in\K(\G^\infty(U),F)$.

$(2)\Leftrightarrow (3)$: Applying Schauder's theorem \cite[Theorem 3.4.15]{Meg-98} and \cite[Proposition 3.4.10]{Meg-98}, we have
\begin{align*}
T_f\in\K(\G^\infty(U),F)&\Leftrightarrow (T_f)^*\in\K(F^*,\G^\infty(U)^*)\\
&\Leftrightarrow f^t=J^{-1}_U\circ(T_f)^*\in\K(F^*,\H^\infty(U)),
\end{align*}

$(2)\Leftrightarrow (4)$: Similarly, one obtains
\begin{align*}
T_f\in\K(\G^\infty(U),F)&\Leftrightarrow (T_f)^*\in\L((F^*,bw^*);\G^\infty(U)^*)\\
&\Leftrightarrow f^t=J^{-1}_U\circ(T_f)^*\in\L((F^*,bw^*);\H^\infty(U)),
\end{align*}
by \cite[Theorem 3.4.16]{Meg-98}, where $bw^*$ denotes the bounded weak* topology.

$(4)\Leftrightarrow (5)\Leftrightarrow (6)$ follows directly from \cite[Proposition 3.1]{Kim-13}.
\end{proof}

Next, we identify $\H^\infty_\K(U,F)$ with the subspace of $\L((F^*,w^*);(\G^\infty(U),w^*))$ consisting of all bounded-weak*-to-norm continuous linear operators from $F^*$ into $\H^\infty(U)$.

\begin{proposition}\label{cor-4-1}
Let $U$ be an open subset of a complex Banach space $E$ and let $F$ be a complex Banach space. The mapping $f\mapsto f^t$ is an isometric isomorphism from $\H_\K^\infty(U,F)$ onto $\L((F^*,bw^*);\H^\infty(U))$.
\end{proposition}

\begin{proof}
Let $f\in\H_\K^\infty(U,F)$. Then $f^t\in\L((F^*,bw^*);\H^\infty(U))$ by Theorem \ref{teo-4-2} and $||f^t||=\left\|f\right\|_\infty$ by Proposition \ref{prop-A}. To prove the surjectivity, take $T\in\L((F^*,bw^*);\H^\infty(U))$. Then $J_U\circ T\in\L((F^*,bw^*);\G^\infty(U)^*)$. If $Q_{\G^\infty(U)}$ denotes the natural injection from $\G^\infty(U)$ into $\G^\infty(U)^{**}$, then $Q_{\G^\infty(U)}(\gamma)\circ J_U\circ T\in\L((F^*,bw^*);\C)$ for all $\gamma\in\G^\infty(U)$ and, by \cite[Theorem 2.7.8]{Meg-98}, $Q_{\G^\infty(U)}(\gamma)\circ J_U\circ T\in\L((F^*,w^*);\C)$ for all $\gamma\in\G^\infty(U)$, that is, $J_U\circ T\in\L((F^*,w^*);(\G^\infty(U)^*,w^*))$ by \cite[Corollary 2.4.5]{Meg-98}. Hence $J_U\circ T=S^*$ for some $S\in\L(\G^\infty(U);F)$ by \cite[Theorem 3.1.11]{Meg-98}. Note that $S^*\in\L((F^*,bw^*);\G^\infty(U)^*)$ and this means that $S\in\K(\G^\infty(U),F)$ by \cite[Theorem 3.4.16]{Meg-98}. Now, $S=T_f$ for some $f\in\H^\infty_\K(U,F)$ by Theorem \ref{teo1} (3) and Theorem \ref{teo-4-2}. Finally, we have $T=J^{-1}_U\circ S^*=J^{-1}_U\circ (T_f)^*=f^t$.
\end{proof}

\subsection{Bounded holomorphic mappings with separable range}

We denote by $\H^\infty_{\S}(U,F)$ the space of all mappings $f\in\H^\infty(U,F)$ such that $f(U)$ is separable. Clearly, $\H^\infty_{\K}(U,F)$ is contained in $\H^\infty_{\S}(U,F)$.

\begin{theorem}\label{teo-3-4-1}
Let $U$ be an open subset of a complex Banach space $E$, let $F$ be a complex Banach space and let $f\in\H^\infty(U,F)$. Consider the following assertions:
\begin{enumerate}
	\item $f\colon U\to F$ has separable range.
	\item $T_f\colon\G^\infty(U)\to F$ has separable range.
	\item $(T_f)^*\colon F^*\to\G^\infty(U)^*$ has separable range.
	\item $f^t\colon F^*\to\H^\infty(U)$ has separable range.
\end{enumerate}
Then $(1)\Leftrightarrow(2)\Leftarrow(3)\Leftrightarrow(4)$.
\end{theorem}

\begin{proof}
$(1)\Leftrightarrow(2)$: We next use some inclusions obtained in the proof of Theorem \ref{teo-3-0}. If $f(U)$ is separable, then so is $\overline{\lin}(f(U))$, and since $T_f(\G^\infty(U))\subseteq\overline{\lin}(f(U))$, we deduce that $T_f(\G^\infty(U))$ is separable. Conversely, if $T_f(\G^\infty(U))$ is separable, then so is $f(U)$ since $f(U)\subseteq\lin(f(U))\subseteq T_f(\G^\infty(U))$.

$(3)\Leftrightarrow(4)$: We apply that the space $\S(E,F)$ of all bounded linear operators between Banach spaces having separable range is an operator ideal (see \cite[1.8.2]{Pie-80}).

$(3)\Rightarrow(2)$: We apply that if $T\in\L(E;F)$ and $T^*\in\L(F^*;E^*)$, then $T\in\S(E,F)$ (see \cite[4.4.8]{Pie-80}).
\end{proof}

\subsection{Approximable bounded holomorphic mappings}

We now enlarge the set of bounded finite-rank holomorphic mappings as follows. We say that a mapping $f\in\H^\infty(U,F)$ is \textit{approximable} if it is the limit in the uniform norm of a sequence of mappings of $\H_\F^\infty(U,F)$. The set of such mappings will be denoted by $\H^\infty_{\overline{\F}}(U,F)$.

Next, we see that every approximable mapping $f\in\H^\infty(U,F)$ has relatively compact range.

\begin{proposition}
Let $U$ be an open subset of a complex Banach space $E$ and let $F$ be a complex Banach space. Then $\H^\infty_{\overline{\F}}(U,F)\subseteq\H_\K^\infty(U,F)$.
\end{proposition}

\begin{proof}
Let $f\in\H^\infty_{\overline{\F}}(U,F)$. Hence there is a sequence $(f_n)_{n\in\N}$ in $\H_\F^\infty(U,F)$ such that $\left\|f_n-f\right\|_\infty\to 0$ as $n\to\infty$. Since $T_{f_n}\in\F(G^\infty(U),F)$ by Theorem \ref{teo-3-0}, $\F(G^\infty(U),F)\subseteq\K(\G^\infty(U),F)$ and $\left\|T_{f_n}-T_f\right\|=\left\|T_{f_n-f}\right\|=\left\|f_n-f\right\|_\infty$ for all $n\in\N$, we deduce that $T_f\in\K(\G^\infty(U),F)$ by \cite[Corollary 3.4.9]{Meg-98}, and so $f\in\H_\K^\infty(U,F)$ by Theorem \ref{teo-4-2}.
\end{proof}

An application of the principle of local reflexivity obtained by C. V. Hutton \cite{Hut-74} shows that an operator $T\in\L(E;F)$ can be approximated by bounded finite-rank linear operators from $E$ into $F$ if and only if $T^*$ can be approximated by bounded finite-rank linear operators from $F^*$ into $E^*$. We now invoke Hutton's theorem to obtain the following.

\begin{theorem}\label{cor-4-0}
Let $U$ be an open subset of a complex Banach space $E$, let $F$ be a complex Banach space and $f\in\H^\infty(U,F)$. The following are equivalent:
\begin{enumerate}
	\item $f\colon U\to F$ can be approximated by bounded finite-rank holomorphic mappings.
	\item $T_f\colon\G^\infty(U)\to F$ can be approximated by bounded finite-rank linear operators.
	\item $f^t\colon F^*\to\H^\infty(U)$ can be approximated by bounded finite-rank linear operators.
\end{enumerate}
\end{theorem}

\begin{proof}
We have
\begin{align*}
f\in\H^\infty_{\overline{\F}}(U,F)
&\Leftrightarrow T_f\in\overline{\F}(\G^\infty(U),F)\\
&\Leftrightarrow (T_f)^*\in\overline{\F}(F^*,\G^\infty(U)^*)\\
&\Leftrightarrow f^t=J^{-1}_U\circ(T_f)^*\in\overline{\F}(F^*,\H^\infty(U))
\end{align*}
by Theorem \ref{teo1} (3) and Theorem \ref{teo-3-0} for the first equivalence, Hutton’s theorem \cite[Theorem 2.1]{Hut-74} for the second one, and the operator ideal property of $\overline{\F}(E,F)$ for the third one \cite[4.2.2]{Pie-80}.
\end{proof}

\subsection{Holomorphic mappings with relatively weakly compact range}

The Davis--Figiel--Johnson--Pe\l czynski factorization theorem \cite{dfjp} asserts that any weakly compact linear operator between Banach spaces factors through a reflexive Banach space. We now extend this result to holomorphic mappings with relatively weakly compact range and give also the analogs of Gantmacher and Gantmacher--Nakamura theorems for such mappings.

We will denote by $\H^\infty_{\W}(U,F)$ the linear space of all holomorphic mappings from $U$ to $F$ that have relatively weakly compact range. Clearly, $\H^\infty_{\K}(U,F)$ is contained in $\H^\infty_{\W}(U,F)$.

\begin{theorem}\label{teo-3-2}
Let $U$ be an open subset of a complex Banach space $E$, let $F$ be a complex Banach space and let $f\in\H^\infty(U,F)$. The following are equivalent:
\begin{enumerate}
	\item $f\colon U\to F$  has relatively weakly compact range.
	\item $T_f\colon\G^\infty(U)\to F$ is weakly compact.
	\item There exist a reflexive complex Banach space $G$, an operator $T\in\L(G;F)$ and a mapping $g\in\H^\infty(U,G)$ such that $f=T\circ g$.
	\item $f^t\colon F^*\to\H^\infty(U)$ is weakly compact.
	\item $f^t\colon F^*\to\H^\infty(U)$ is weak*-to-weak continuous.
\end{enumerate}
\end{theorem}

\begin{proof}
$(1)\Leftrightarrow (2)$ (\cite[Proposition 3.1 (b)]{Muj-91}): It follows with a similar proof to that of the equivalence $(1)\Leftrightarrow (2)$ of Theorem \ref{teo-4-2} by taking into account that the norm closure and weak closure of the convex hull of a subset of a normed space coincide \cite[Corollary 2.5.18]{Meg-98} and that the norm-closed convex hull of a weakly compact subset of a Banach space is itself weakly compact \cite[Theorem 2.8.14]{Meg-98}.

$(2)\Rightarrow(3)$: Applying the Davis--Figiel--Johnson--Pe\l czynski theorem, there exist a reflexive complex Banach space $G$ and operators $T\in\L(G;F)$ and $S\in\L(\G^\infty(U);G)$ such that $T_f=T\circ S$. Taking $g:=S\circ\delta_U\in\H^\infty(U,G)$, we conclude that $f=T_f\circ\delta_U=T\circ S\circ\delta_U=T\circ g$.

$(3)\Rightarrow(1)$: $f(U)=T(g(U))$ is relatively weakly compact because $T$ is weak-to-weak continuous by \cite[Theorem 2.5.11]{Meg-98} and $g(U)$ is relatively weakly compact in $G$ since it is a bounded subset of the reflexive Banach space $G$ (see \cite[Theorem 2.8.2]{Meg-98}).

$(2)\Leftrightarrow (4)$: We have
\begin{align*}
T_f\in\W(\G^\infty(U),F)&\Leftrightarrow (T_f)^*\in\W(F^*,\G^\infty(U)^*)\\
&\Leftrightarrow f^t=J^{-1}_U\circ(T_f)^*\in\W(F^*,\H^\infty(U)),
\end{align*}
by Gantmacher's theorem \cite[Theorem 3.5.13]{Meg-98} and \cite[Proposition 3.5.11]{Meg-98}.

$(2)\Leftrightarrow (5)$: We have
\begin{align*}
T_f\in\W(\G^\infty(U),F)&\Leftrightarrow (T_f)^*\in\L((F^*,w^*);(\G^\infty(U)^*,w))\\
&\Leftrightarrow f^t=J^{-1}_U\circ(T_f)^*\in\L((F^*,w^*);(\H^\infty(U),w))
\end{align*}
by Gantmacher--Nakamura's theorem \cite[Theorem 3.5.14]{Meg-98} and \cite[Corollary 2.5.12]{Meg-98}.
\end{proof}

We next identify $\H^\infty_{\W}(U,F)$ with the subspace of $\L((F^*,w^*);(\G^\infty(U),w^*))$ formed by all weak*-to-weak continuous linear operators from $F^*$ into $\H^\infty(U)$.

\begin{proposition}\label{cor-4-1b}
Let $U$ be an open subset of a complex Banach space $E$ and let $F$ be a complex Banach space. The mapping $f\mapsto f^t$ is an isometric isomorphism from $\H^\infty_{\W}(U,F)$ onto $\L((F^*,w^*);(\H^\infty(U),w))$.
\end{proposition}

\begin{proof}
In view of Theorem \ref{teo-3-2} and Proposition \ref{prop-A}, we only need to show that the mapping in the statement is surjective. Let $T\in\L((F^*,w^*);(\H^\infty(U),w))$. Then $J_U\circ T\in\L((F^*,w^*);(\G^\infty(U)^*,w))$ by \cite[Theorem 2.5.11]{Meg-98}, and this last set is contained in $\L((F^*,w^*);(\G^\infty(U)^*,w^*))$ since the weak topology is stronger than the weak* topology on the dual of a normed space. 
It follows that $J_U\circ T=S^*$ for some $S\in\L(\G^\infty(U);F)$ by \cite[Theorem 3.1.11]{Meg-98}. Hence $S^*\in\L((F^*,w^*);(\G^\infty(U)^*,w))$ and, by Gantmacher--Nakamura's theorem \cite[Theorem 3.5.14]{Meg-98}, $S\in\W(\G^\infty(U),F)$. Now, $S=T_f$ for some $f\in\H_{\W}^\infty(U,F)$ by Theorem \ref{teo1} (3) and Theorem \ref{teo-3-2}. Finally, $T=J^{-1}_U\circ S^*=J^{-1}_U\circ (T_f)^*=f^t$, as desired.
\end{proof}

\subsection{Bounded holomorphic mappings with Rosenthal range}

Given Banach spaces $E$ and $F$, a set $A\subseteq E$ is called \textit{Rosenthal (or conditionally weakly compact)} if every sequence in $A$ admits a weak Cauchy subsequence, and a bounded linear operator $T\colon E\to F$ is called a \textit{Rosenthal operator} if $T(B_E)$ is a Rosenthal subset of $F$. It is known (see, for example, \cite{Alv-88}) that $T\colon E\to F$ is Rosenthal if and only if it factors through a Banach space that does not contain an isomorphic copy of $\ell_1$.

We now prove a similar result for holomorphic mappings with Rosenthal range, and the Rosenthal factorization theorem then allows us to factor these mappings through a Banach space not containing an isomorphic copy of $\ell_1$. We denote by $\H^\infty_{\R}(U,F)$ the linear space of all bounded holomorphic mappings from $U$ to $F$ that have Rosenthal range. Clearly, $\H^\infty_{\W}(U,F)$ is contained in $\H^\infty_{\R}(U,F)$.

\begin{theorem}\label{teo-3-3}
Let $U$ be an open subset of a complex Banach space $E$, let $F$ be a complex Banach space and let $f\in\H^\infty(U,F)$. The following are equivalent:
\begin{enumerate}
	\item $f\colon U\to F$ has Rosenthal range.
	\item $T_f\colon\G^\infty(U)\to F$ is Rosenthal.
	\item There exist a complex Banach space $G$ which does not contain an isomorphic copy of $\ell_1$, an operator $T\in\L(G;F)$ and a mapping $g\in\H^\infty(U,G)$ such that $f=T\circ g$.
\end{enumerate}
\end{theorem}

\begin{proof}
$(1)\Leftrightarrow(2)$: It can be proved similarly as the same equivalence in Theorem \ref{teo-4-2} taking into account that the norm-closed absolutely convex hull of a Rosenthal subset of a Banach space is itself Rosenthal (see \cite[p. 357]{Lin-89}).

$(2)\Rightarrow(3)$: Rosenthal factorization theorem (see e.g. \cite{Alv-88}) gives a complex Banach space $G$ not containing an isomorphic copy of $\ell_1$ and operators $T\in\L(G;F)$ and $S\in\L(\G^\infty(U);G)$ such that $T_f=T\circ S$. If we write $g:=S\circ\delta_U\in\H^\infty(U,G)$, then $f=T_f\circ\delta_U=T\circ S\circ\delta_U=T\circ g$.

$(3)\Rightarrow(1)$: Notice that $f(U)=T(g(U))$ where $T$ is weak-to-weak continuous \cite[Theorem 2.5.11]{Meg-98} and $g(U)$ is Rosenthal in $G$ by Rosenthal's $\ell_1$-theorem \cite{Ros-74}. 
\end{proof}

\subsection{Holomorphic mappings with Asplund range}

By \cite[Definition 5.1.2]{Bou-83}, a bounded set $D\subseteq E$ is said to be \textit{Asplund} or to have the  \textit{Asplund property} if every convex continuous function $f\colon E\to\mathbb{R}$ is $D$-differentiable on a residual subset of $E$.
A Banach space $E$ is called an \textit{Asplund space} if every convex continuous function $f\colon E\to\R$ is Fr\'echet differentiable on a dense $G_\delta$ set in $E$. This definition is due to E. Asplund \cite{Asp-68} under the name \textit{strong differentiability space}.  We refer to R. D. Bourgin \cite[Theorem 5.2.11]{Bou-83} and R. R. Phelps \cite[Theorem 2.34]{Phe-93} for some equivalent formulations of the concept of Asplund set in a Banach space. In particular, we remark that the Asplund spaces are the Banach spaces for which each separable subspace has a separable dual. We will also use the paper \cite{Ste-81} by C. Stegall for the properties of Asplund sets and Asplund operators in Banach spaces (see also Theorem 5.5.4 in \cite{Bou-83}).

Following \cite[Definition 1.2]{Ste-81}, we say that a closed, bounded, convex subset $K$ of a Banach space $E$ has the \textit{Radon--Nikodym property} (RNP) if for any finite measure space $(\Omega,\Sigma,\mu)$, every $m\colon\Sigma\to E$ that is $\mu$-continuous, countably additive, of finite variation, with average range $\{\mu(E)^{-1}m(E)\colon E\in\Sigma,\, \mu(E)>0\}\subseteq K$, is representable by a Bochner integrable function. A Banach space $E$ has RNP if every closed, bounded and convex subset of $E$ has RNP. By \cite[Theorem 2.8]{Ste-81}, a Banach space $E$ is an Asplund space if and only if $E^*$ has RNP.

Given two Banach spaces $E$ and $F$, we say that an operator $T\in\L(E;F)$ is \textit{Radon--Nikod\'ym} if $T$ factors through a Banach space $W$ with $W$ having RNP. We denote by $\RN(E,F)$ the linear space of all Radon--Nikod\'ym operators from $E$ into $F$. According to \cite[Theorem 2.11]{Ste-81}, we have that $T\in\AS(E,F)$ if and only if $T^*\in\RN(F^*,E^*)$.

We will denote by $\H^\infty_{\AS}(U,F)$ the linear space of all holomorphic mappings from $U$ to $F$ such that $f(U)$ is an Asplund subset of $F$. Clearly, $\H^\infty_{\W}(U,F)$ is contained in $\H^\infty_{\AS}(U,F)$.

\begin{theorem}\label{teo-3-4}
Let $U$ be an open subset of a complex Banach space $E$, let $F$ be a complex Banach space and let $f\in\H^\infty(U,F)$. The following are equivalent:
\begin{enumerate}
	\item $f\colon U\to F$ has Asplund range.
	\item $T_f\colon\G^\infty(U)\to F$ is Asplund.
	\item There exist a complex Asplund space $G$, an operator $T\in\L(G;F)$ and a mapping $g\in\H^\infty(U,G)$ such that $f=T\circ g$.
	\item $f^t\colon F^*\to\H^\infty(U)$ is Radon--Nikod\'ym.
\end{enumerate}
\end{theorem}

\begin{proof}
$(1)\Leftrightarrow(2)$: Since $f(U)$ is an Asplund set and $T_f(B_{\G^\infty(U)})\subseteq\overline{\aco}(f(U))$, then $T_f(B_{\G^\infty(U)})$ is an Asplund set by Lemmas 1.3 and 1.4 in \cite{Ste-81} and Theorem 5.5.4 in \cite{Bou-83}.

$(2)\Rightarrow(3)$: If $T_f\in\AS(\G^\infty(U),F)$, then Theorem 2.11 in \cite{Ste-81} assures that $T_f=T\circ S$ with $S\in\L(\G^\infty(U);G)$ and $T\in\L(G;F)$, where $G$ is a complex Asplund space. Taking $g:=S\circ\delta_U\in\H^\infty(U,G)$, we have $f=T_f\circ\delta_U=T\circ S\circ\delta_U=T\circ g$.

$(3)\Rightarrow(1)$: $f(U)=T(g(U))$ is an Asplund set by Theorem 2.8, Theorem 1.13 and Lemma 1.7 in \cite{Ste-81}.

$(2)\Rightarrow(4)$: By \cite[Theorem 2.11]{Ste-81}, \cite[24.2]{Pie-80} and Proposition \ref{prop-A}, we have
\begin{align*}
T_f\in\AS(\G^\infty(U),F)&\Leftrightarrow (T_f)^*\in\RN(F^*,\G^\infty(U)^*)\\
&\Leftrightarrow f^t=J^{-1}_U\circ(T_f)^*\in\RN(F^*,\H^\infty(U)).
\end{align*}
\end{proof}

Combining Theorem \ref{teo1} (3) with the results stated above, we see that the isometric isomorphism $f\mapsto T_f$ from $\H^\infty(U,F)$ onto $\L(\G^\infty(U);F)$ induces the following identifications.

\begin{corollary}\label{cor-messi}
Let $U$ be an open subset of a complex Banach space $E$ and let $F$ be a complex Banach space. The mapping $f\mapsto T_f$ is an isometric isomorphism from $\H^\infty_{\I}(U,F)$ onto $\I(\G^\infty(U),F)$ in the cases $\I=\F,\K,\S,\overline{\F},\W,\R,\AS$. $\hfill\Box$
\end{corollary}

\textbf{Acknowledgements.} The authors are grateful to the referees for their helpful comments that have improved this paper.

\end{document}